\documentclass[a4paper, 10pt]{article}				
\usepackage[latin1]{inputenc}				
\usepackage[T1]{fontenc}						
\usepackage{lmodern}					
\usepackage[english]{babel}					
\usepackage{amsmath,amsfonts}			
\usepackage{amssymb}					                      				 
\usepackage{multirow}		
\usepackage{graphicx}
\usepackage{bm}							
\usepackage[standard]{ntheorem}
\usepackage{enumerate}
\usepackage{fancyhdr}
\usepackage{epigraph}

\usepackage[export]{adjustbox}
\usepackage{float}

\usepackage{verbatim}
\usepackage{color}
\usepackage{a4wide}

\newcommand{\R}{\mathbb{R}}

\newcommand{\N}{\mathbb{N}}

\newcommand{\Z}{\mathbb{Z}}

\newcommand{\Q}{\mathbb{Q}}

\newcommand{\T}{\mathbb{T}}

\newcommand{\e}{\epsilon}

\newcommand{\h}{\eta}
\newcommand{\kk}{\kappa}

\newcommand{\sub}{\subseteq}

\newcommand{\lav}{\left|}
\newcommand{\rav}{\right|}
\newcommand{\lf}{\left\lfloor}
\newcommand{\rf}{\right\rfloor}
\newcommand{\ldav}{\left| \left|}
\newcommand{\rdav}{\right| \right|}
\newcommand{\ldb}{\left[\!\left[}
\newcommand{\rdb}{\right]\!\right]}
\newcommand{\oo}{\infty}
\newcommand{\lc}{\left(} 
\newcommand{\rc}{\right)}
\newcommand{\lb}{\left[}
\newcommand{\rb}{\right]}
\newcommand{\lfp}{\left\{}
\newcommand{\rfp}{\right\}}
\newcommand{\mc}{\mathcal}

\numberwithin{Theor}{section}

\numberwithin{theorem}{section}
\numberwithin{proposition}{section}
\numberwithin{lemma}{section}
\numberwithin{Conjecture}{section}
\numberwithin{definition}{section}
\numberwithin{Corollary}{section}
\numberwithin{Remark}{section}
\numberwithin{note}{section}
\numberwithin{Example}{section}

\title{Density of Oscillating Sequences in the Real Line}
\author{Ioannis Tsokanos}
\date{ }

\begin{document}

\maketitle

\begin{abstract}
In this paper we study the density in the real line of oscillating sequences of the form
$$\lc g(k)\cdot F(k\alpha) \rc_{k \in \N} ,$$
where $g$ is a positive increasing function and $F$  a real continuous 1-periodic function.
This extends work by Berend, Boshernitzan and Kolesnik [Distribution Modulo $1$ of Some Oscillating Sequences I-III] who established differential properties on the function $F$ ensuring that the oscillating sequence is dense modulo $1$. 

More precisely, when $F$ has finitely many  roots in $[0,1)$, we provide necessary and also sufficient conditions for the oscillating sequence under consideration to be dense in $\R$.  All the results are stated in terms of the Diophantine properties of $\alpha$,  with the help of the theory of continued fractions. 
\end{abstract}

\par %%use paragraph \par not \\ \\

\section{Introduction}

Given a real number $x$, denote by $\lfp x \rfp_{2}$ the signed fractional part of $x$, which is  the unique  real number in $\lb -{1 \over 2}, {1 \over 2} \rc$ such that  $x - \lfp x\rfp_{2} \in \Z$. Similarly, $\lfp x \rfp $ stands for the fractional part of $x$.  Denote by $\ldav x \rdav $ its distance from the nearest integer: $ \ldav x \rdav =  \min_{n \in \Z} \lav x - n \rav .$
Let also $\ldb a,b \rdb$ be the integer interval with end points determined by the real numbers $a$ and $b$; that is $ \ldb a,b \rdb = \lfp n \in \Z: a \le n \le b \rfp .$  
Finally, we will make use of Landau's little-$o$ notation: if, given two real functions $w,v:\R^{+}\mapsto \R^{+}$, it holds that $ w(x)/v(x) \to 0$ as $x \to +\oo$ (respectively, as $x\to 0$), then one may write  $w(x) = o\lc v(x) \rc$ as $x \to +\oo$ (respectively, as $x\to 0$).

It is asked in \cite{MathStackExchangeDensityOfOscillatingSequences} whether the sequence 
$ \lc k \cdot \sin\lc k \rc \rc_{k \in \N} $ is dense in $\R$. More generally, it is natural to determine the values of the parameters $\beta \;>0$ and $\alpha \in \R$  for which the oscillating sequence $\lc k^{\beta}\cdot \sin\lc 2\pi\cdot k\alpha \rc \rc_{k \in \N} $ 
is dense in $\R$. In this paper we answer this question by studying the density properties in $\R$ of the more general class of oscillating sequences of the form 
\begin{equation}\label{eqOscillatingSequences}
\lc g(k)\cdot F\lc k\alpha \rc \rc_{k \in \N} , 
\end{equation}
where
\begin{equation}\label{eqGrowthRate}
g(t) = t^{\beta} + o\lc t^{\beta} \rc \quad \text{as } t \to +\oo
\end{equation}
for some $\beta \;> 0 $, and where the function $F$ is a real, $1$-periodic, continuous function with only isolated roots. We assume further that, if $r \in \R$ is a root of $F$, then $F$ has the form  
\begin{equation}\label{eqPeriodicFunction}
F\lc r+x \rc =  c_{r} \cdot \e(x) \cdot \lav x \rav^{\gamma(r)} + o\lc \lav x \rav^{\gamma(r)} \rc \quad \text{as } x \to 0 
\end{equation}
for some $\gamma(r) \;> 0 $ and some $c_{r} \in \R\backslash \lfp 0 \rfp$. 
Here, the function $\e: \R \mapsto \lfp -1, 0, 1 \rfp $ stands for the sign function 
$$ \e(x) = 
\begin{cases}
1, \quad \text{if }  x \;> 0 \\
0, \quad \text{if }  x = 0 \\
-1, \quad \text{if } x \;< 0 .
\end{cases} $$

A study of the density of oscillating sequences in the torus $\T = \R/\Z$ has been made by Berend, Boshernitzan and Kolesnik (see \cite{BerendDistributionModulo1ofsomeOscillatingSequences1,BerendDistributionModulo1ofsomeOscillatingSequences2,BerendDistributionModulo1ofsomeOscillatingSequences3}). In this body of work, the authors consider oscillating sequences of the form 
\begin{equation}\label{eqToralOscillatingSequences} 
\lc P(k)\cdot f\lc Q(k) \rc \rc_{k \in \N} ,
\end{equation}
where $P,Q$ are polynomials  and $f$  is a (highly differentiable) periodic function with period $T \;>0$. In particular, they consider three aspects of the problem: the problem of small values modulo $1$ of such sequences, that of their density modulo $1$, and that of  their uniform distribution.

More precisely, in \cite{BerendDistributionModulo1ofsomeOscillatingSequences1}, the authors deal with the above-stated problems by providing in each case sufficient conditions on the degree of differentiability of the function $f$ at the point $Q(0)$ for the related properties to hold. In \cite{BerendDistributionModulo1ofsomeOscillatingSequences2}, they generalise the results regarding the small values and the density of the sequence \eqref{eqToralOscillatingSequences} in two directions. On the one hand, they allow the function $f$ to  be quasi-periodic, that is $f(x) = f_{0}(x,x,...,x),$ where $f_{0}: \R^{n} \to \R$ is a periodic function of several variables. On the other hand, they study  a more general family of sequences, namely sequences of the form $\lc P(k)\cdot f\lc Q(k) \rc \cdot g\lc R(k) \rc \rc_{k \in \N}$, where $R(k)$ is a polynomial and the function $g$ is periodic.  For instance, they prove that, given integers $d$ and $l$, there exists  $r = r(d,l)$ having the following properties: for any polynomial $P$ of degree $d$, any function $f$ with $f^{(s)}(0) \neq 0$ for some $s\ge r$ and any real number $\alpha$ with ${\alpha \over T}$ irrational, the sequence $\lc P(k)\cdot f\lc k^{l}\cdot {\alpha \over T} \rc \rc_{k \in \N}$ is dense modulo $1$.

Other results regarding the distribution of the sine function in the real line are given for instance in \cite{AdiceamRationalApprximationArithmeticProgressions}. In this paper, Adiceam exploits a result concerning rational approximations of irrationals with the numerators and denominators of the rational approximants restricted to prescribed arithmetic progressions, and proves that  for every $\rho \in \R$ and irrational $\alpha$,  it holds that
$$ \limsup_{k \to +\oo}\lc \sin\lc 2\pi k\alpha + \rho \rc \rc^{k} = -\liminf_{k \to +\oo}\lc \sin\lc 2\pi k\alpha + \rho \rc \rc^{k} = 1 .$$

Our approach to study the sequence \eqref{eqOscillatingSequences} makes a connection between its density properties in $\R$ and the density properties of auxiliary sequences of the form 
\begin{equation}\label{eqFractionalParts}
\lc k^{\beta}\cdot \lfp k\alpha - \rho \rfp_{2} \rc_{k \in \N} ,
\end{equation}
where $\rho$ is a real number (see Proposition \ref{PropOscillatingSequencesAndFractionalParts} in Section \ref{SectionAuxiliaryResults} below for details). Working with the signed fractional part instead of the distance from the nearest integer, which may seem more natural, is a consequence of working in the real line as we have to consider separately 	the positive and the negative values of the function \eqref{eqPeriodicFunction}.

Little seems to be known regarding the  density of oscillating sequences in  the real line. One of the goals of this paper is to relate the density of \eqref{eqOscillatingSequences} with the approximation properties of $\alpha$. Here, by approximation properties  we are referring to the irrationality measure  $\mu(\alpha)$ of $\alpha$: 
$$ \mu(\alpha) = \sup\lfp
\begin{split}
 v \;> 0 : \lav \alpha - {p\over q} \rav \le {1 \over q^{v} } \text{ holds for infinitely many rationals ${p\over q}$ } \\  \text{with $ \gcd(p,q) =1 $} \end{split}\rfp .$$
It can readily be checked that every rational number $r $ has irrationality measure $\mu(r) = 1$ while, from Dirichlet's theorem in Diophantine approximation, for every irrational $x$  it holds that $\mu(x) \ge 2$. 
We consider more precisely some additional quantities which refine the notion of irrationality measure. To define them, we first recall the continued fraction expansion and 
the Ostrowski expansion  of a real number. 
Throughout this paper, the continued fraction expansion \cite[Section 3.1]{BeresnevichSumsOfReciprocals} of every $\alpha \in \R\backslash \Q$  is denoted by 
$$ \alpha = \lb a_{0};a_{1},...,a_{n}...\rb = a_{0}   +\cfrac{1}{ a_{1} +\cfrac{1}{ a_{2} + \cfrac{1}{ a_{3} + \dots}}}      $$
and the sequence of the denominators of the convergents of $\alpha$ by $\lc q_{n} \rc_{n \in \N}$. 
Given an irrational $\alpha= \lb a_{0};a_{1},a_{2},...\rb$ and a real number $\rho$, the Ostrowski expansion \cite[Section 3.2, Lemma 3.2]{BeresnevichSumsOfReciprocals} of $\rho$ with base $\alpha$ is the unique choice of natural numbers $\lfp e_{n}(\rho)\rfp_{n \in \N_{0}}$ and of an integer $\rho_{0}$ such that 
\begin{equation}\label{eqOstrowskiExpansion}
\rho  = \rho_{0} +  e_{0}(\rho)\cdot \lfp \alpha \rfp + \sum_{n=1}^{+\oo} e_{n}(\rho)\cdot \lfp q_{n}\alpha \rfp_{2} ,
\end{equation}
where $ \rho - \rho_{0} \in \lb -\alpha, 1-\alpha \rc$, $ e_{0}(\rho)\in\ldb 0, a_{1} -1 \rdb$ and $ e_{n}(\rho) \in \ldb 0 ,a_{n+1} \rdb $ for every $ n \ge 1$,
with $ e_{n}(\rho) = 0$ whenever $ e_{n+1}(\rho) = a_{n+2}$ for $ n \ge 1$.

\begin{definition}[Signed Irrationality Evaluation]\label{DefSignedIrrationalityEvaluation}
	Given an irrational number $\alpha \in \R\backslash\Q$,  a positive real number $\beta \;> 0$ and a real number $\rho\in \R$, denote by $\mu_{+}\lc \alpha,\beta,\rho \rc$ and $\mu_{-}\lc \alpha,\beta,\rho \rc$  the quantities 
	$$ \mu_{+}\lc \alpha,\beta,\rho\rc = \liminf_{\substack{k \to +\oo,\\ \lfp k\alpha - \rho \rfp_{2} \;> 0} } k^{\beta}\cdot \lfp k\alpha -\rho \rfp_{2} \ge 0   $$
	and 
	$$ \mu_{-}\lc \alpha,\beta,\rho\rc = \liminf_{\substack{k \to +\oo,\\ \lfp k\alpha - \rho \rfp_{2} \;< 0} } - k^{\beta}\cdot \lfp k\alpha -\rho \rfp_{2} \ge 0   .$$
	Moreover, denote by $\mu(\alpha,\beta,\rho)  = \min\lfp \mu_{+}(\alpha,\beta,\rho), \mu_{-}(\alpha,\beta,\rho) \rfp$ the minimum of the above two quantities. 	When $\rho=0$, we may write $\mu_{+}(\alpha,\beta),\mu_{-}(\alpha,\beta)$ and $\mu(\alpha,\beta) $ to simplify notation.
\end{definition}

Given the Ostrowski expansion \eqref{eqOstrowskiExpansion} of $\rho$, set further 
 $$
 \tau_{+}(\alpha,\beta,\rho) =  
 \liminf_{ \substack{ n \to +\oo }} \max\lfp 1, \min\lfp e_{2n}(\rho)^{\beta}, \lc a_{2n+1} - e_{2n}(\rho) \rc^{\beta +1 \over 2}\rfp \rfp \cdot q_{2n}^{\beta}\cdot \lfp q_{2n}\alpha \rfp_{2}  \ge 0
 $$
and
$$  \tau_{-}(\alpha,\beta,\rho) =
  \liminf_{ \substack{ n \to +\oo }} - \max\lfp 1, \min\lfp e_{2n-1}(\rho)^{\beta}, \lc a_{2n} - e_{2n-1}(\rho) \rc^{\beta +1 \over 2}\rfp \rfp \cdot q_{2n-1}^{\beta}\cdot \lfp q_{2n-1}\alpha \rfp_{2}  \ge 0   .  $$

Our main result provides necessary and also sufficient conditions on the oscillating sequence \eqref{eqOscillatingSequences} to be dense in $\R$.

\begin{theorem}\label{TheorOscillatingSequence}
	Denote by $\lc y_{k} \rc_{k \in \N}$ the sequence defined in \eqref{eqOscillatingSequences}. Let the function $F$ satisfy assumption \eqref{eqPeriodicFunction} and let $g$ satisfy assumption \eqref{eqGrowthRate}.
	\begin{enumerate}
		\item  If the sequence \eqref{eqOscillatingSequences} is dense in $\R^{+}$ (resp. in $\R^{-}$) then there exists a root $r$ of $F$ such that either  $c_{r} \;>0 $ (resp. $c_{r} \;< 0$) and $\mu_{+}\lc \alpha,{\beta \over \gamma(r)}, r\rc =0 $,  or else  $c_{r}\;< 0$ (resp. $c_{r} \;>0$)  and $\mu_{-}\lc \alpha,{\beta \over \gamma(r)}, r\rc =0$. Moreover, if the root $r$ is rational then this condition is also sufficient. 
		
		\item If there exists a root $r$ of $F$ such that either $c_{r} \;>0 $ (resp. $c_{r}\;< 0$) and $\tau_{+}\lc \alpha,{\beta \over \gamma(r)}, r \rc = 0$, or else $c_{r} \;<0 $ (resp. $c_{r}\;> 0$) and $\tau_{-}\lc \alpha,{\beta \over \gamma(r)}, r \rc = 0$, then the sequence \eqref{eqOscillatingSequences} is dense in $\R^{+}$ (resp. in $\R^{-}$).		
	\end{enumerate}
\end{theorem}

Under the assumptions of Theorem \ref{TheorOscillatingSequence}, the density of the oscillating sequence \eqref{eqOscillatingSequences} depends only on the local properties of $F$ around its isolated roots. In order to establish Theorem \ref{TheorOscillatingSequence}, we first prove the  results for the auxiliary sequence \eqref{eqFractionalParts}. Thus, in Section \ref{SectionRationalValues}, we prove that if $\rho$ is rational, then the sequence \eqref{eqFractionalParts} is dense in $\R^{+}$ (resp. $\R^{-}$) if and only if $\mu_{+}\lc \alpha,\beta, \rho \rc = 0$ (resp. $\mu_{-}\lc \alpha,\beta, \rho \rc = 0$). In Section \ref{SectionRealValues}, we will exploit the Ostrowski expansion in order to prove that, if $\tau_{+}\lc \alpha,\beta, \rho\rc =0 $ (resp. $\tau_{-}\lc \alpha,\beta, \rho\rc =0 $), then the squence \eqref{eqFractionalParts} is dense in $\R^{+}$ (resp. in $\R^{-}$).

In the special case where $F(x) = \sin\lc 2\pi \cdot x \rc$, we obtain the following  corollary answering the opening question of the paper.

\begin{corollary}\label{CorTrigonometricOscillatingSequence}
	Given $\beta \;>0$  and $\alpha \in \R\backslash \Q$, the sequence
	$$ \lc k^{\beta}\cdot \sin\lc 2\pi\cdot k\alpha \rc \rc_{k \in \N} $$
	is dense in $\R$  if and only if at least one of the following holds:
	\begin{enumerate}
		\item $\mu_{+}(\alpha,\beta) = 0$ and $\mu_{-}(\alpha,\beta) = 0$,
		\item $\mu_{+}\lc \alpha,\beta\rc = 0 $ and $\mu_{+}\lc \alpha,\beta,{1 \over 2} \rc  =0 $,
		\item $\mu_{-}(\alpha,\beta) = 0$ and $\mu_{-}\lc \alpha,\beta,{1\over 2} \rc = 0 $.
	\end{enumerate}	
	
\end{corollary}

For instance, we can apply Corollary \ref{CorTrigonometricOscillatingSequence} when $\alpha$ is badly approximable; that is, when there exists $c \;> 0$  such that for every $k \in \N$, it holds $k\cdot \ldav k\alpha \rdav \ge c$. In this case,  for every $\beta \ge 1$, it holds that  $\mu(\alpha,\beta) \;>0 $ and therefore $\lc k^{\beta}\cdot \sin\lc 2\pi \cdot k\alpha \rc \rc_{k \in \N}$ is not dense in $\R$. Similarly, if $\beta \;<1$ it holds that $\mu_{\pm}(\alpha,\beta) = 0$ and the same sequence is dense in $\R$.

\begin{remark} From Definition \ref{DefSignedIrrationalityEvaluation}, it follows immediately that
\linebreak  $ \mu(\alpha,\beta,r) = \underset{k\to +\oo}{\liminf} \text{ } k^{\beta}\cdot \ldav k\alpha -r \rdav .$  However more natural this quantity may seem, as proved in Section \ref{SectionMainResults}, it does not hold that $\mu_{+}(\alpha,\beta,r) = 0$ if and only if $\mu_{-}(\alpha,\beta,r) =0$. This is the reason why the results are not stated in terms of the quantity $\mu(\alpha,\beta,r)$ alone.
\end{remark}

Theorem \ref{TheorOscillatingSequence} also yields the following corollary stating some cases where the sequence \eqref{eqOscillatingSequences} is trivially dense in $\R$.

\begin{corollary}\label{CorOscillatingSequenceSmallValuesOfBeta}
	Let $\lc y_{k} \rc_{k \in \N}$  be the sequence defined in \eqref{eqOscillatingSequences} with the function $F$ satisfying assumption \eqref{eqPeriodicFunction} and $g$ satisfying assumption \eqref{eqGrowthRate}. If there exists a root $r \in \R$ of $F$ such that ${
		\beta \over \gamma(r)} \in (0,1) $, then the sequence \eqref{eqOscillatingSequences} is dense in $\R$.
\end{corollary}

Note that the sufficient condition stated in Theorem \ref{TheorOscillatingSequence} is not necessary. This is proved in Section \ref{SectionRealValues} by explicitly constructing a suitable sequence $\lc e_{n} \rc_{n \ge 0}$ in the Ostrowski expansion \eqref{eqOstrowskiExpansion}.

In addition to Theorem \ref{TheorOscillatingSequence}, we prove the following result which, in the case where $\rho \in \Q$, characterizes the quantities $\mu_{\pm}(\alpha,\beta,\rho)$ in terms of the sequence of denominators of the convergents to the irrational $\alpha$.

\begin{theorem}\label{TheorIrrationalityEvaluationRationalValues}
	Given $\beta \ge 1$, an irrational number $\alpha \in \R\backslash \Q$  and a rational number $ \theta \in \Q$,  where $ \theta = {p\over q}$ for some $p \in \Z,q \in \N$ with $(p,q)=1$, it holds that
	$$ \mu_{+} \lc \alpha,\beta,{p \over q} \rc = 0 \quad \lc \text{resp.} \quad \mu_{-}\lc \alpha,\beta, {p \over q} \rc = 0 \rc $$
	if and only if
	$$ \liminf_{\substack{ n \to +\oo, \\ q|q_{n} }} q_{2n}^{\beta}\cdot \lfp q_{2n} \alpha \rfp_{2} = 0 \quad \lc \text{resp.} \quad    \liminf_{\substack{ n \to +\oo, \\ q|q_{n} }} q_{2n-1}^{\beta}\cdot \lfp q_{2n-1} \alpha \rfp_{2}    \rc  = 0.  $$ 
	
\end{theorem}

Finally, we provide results regarding the density of oscillating sequences \eqref{eqOscillatingSequences} in $\R$ when the parameters $\alpha$ and $\beta$ satisfy $\mu(\alpha,\beta) = +\oo$.  Note that the inequalities $  \mu_{+}(\alpha,\beta) \le \tau_{+}(\alpha,\beta,\rho) $ and $ \mu_{-}(\alpha,\beta) \le \tau_{-}(\alpha,\beta,\rho) $ hold for every choice of $\alpha,\beta$ and $\rho$ (see Lemma \ref{LemIrrationalityEvaluationRationalValues}). The aforementioned assumption therefore implies that, for every real $\rho$, $\tau_{+}(\alpha,\beta, \rho) = \tau_{-}(\alpha,\beta, \rho) = +\oo$. Thus, the sufficient condition in the statement of Theorem \ref{TheorOscillatingSequence} does not hold.
Before stating the result, recall the definition of inhomogeneous Bohr sets (see \cite{ChowBohrSetsAndMultiplicativeDiophantineApproximation,TaoBohrSets} for more details): given  a real number $\rho$, an irrational number $\alpha$, a natural number $N$ and a positive number $\e \;>0$, let
\begin{equation}\label{eqBohrSets}
\mc{N}_{\rho}\lc N, \alpha, \e \rc = \lfp k \in \N: \text{ } k \le N, \text{ } \ldav k\alpha - \rho \rdav \le \e \rfp .
\end{equation}
We use Bohr sets in order to capture the terms of the sequence \eqref{eqFractionalParts} which affect its density properties in $\R$. Given the Ostrowski expansion \eqref{eqOstrowskiExpansion} of $\rho$, define a sequence of natural numbers by setting
\begin{equation}\label{eqOstrowskiSequence}
\kk_{n} = \sum_{j=0}^{n} e_{j}(\rho)\cdot q_{j} \quad \text{for all } n \ge 0 .
\end{equation} 

\begin{theorem}\label{TheorBohrSets}
	Let $\alpha$ be an irrational number and let $\beta \;>0$ be such that $\mu(\alpha,\beta) = +\oo$.
	Denote by $\lc w_{k} \rc_{k \in \N}$ the sequence defined in \eqref{eqFractionalParts}.
	Let $\rho$ be a real number and let $\lc e_{j}(\rho) \rc_{j \ge 0}$ be the digits in its Ostrowski expansion. Also, let $\lc \kk_{n} \rc_{n\ge 0}$ be the sequence defined in \eqref{eqOstrowskiSequence}. Then, the sequence $\lc w_{k} \rc_{k\in \N}$ is dense in $\R$ if and only if the subsequence 
	$ \lc w_{k} \rc_{k \in \mathfrak{D}}$ is dense in $\R$, where 
	\begin{equation}\label{eqTheorBSset}
	\mathfrak{D} = \bigcup_{n=1}^{+\oo} \lc \mc{N}_{\rho}\lc n \rc \cup \mc{N'}_{\rho} \lc n \rc  \rc
	\end{equation}
	with $\mc{N}_{\rho}(n)=  \mc{N}_{\rho}\lc \kk_{n}, \alpha, \ldav q_{n} \alpha \rdav \rc$ and $\mc{N'}_{\rho}(n) =\mc{N}_{\rho} \lc \kk_{n-1}+ q_{n}, \alpha, { \ldav q_{n-1}\alpha \rdav \over 1 + e_{n-1}^{\beta}} \rc$. 
	Moreover, the inclusions
	\begin{equation}\label{eqTheorBSInclusions}
	\lfp \kk_{n} \rfp_{n \in \N} \sub \mathfrak{D}  \quad \text{and}\quad \mathfrak{D}\sub \bigcup_{n=0}^{+\oo} \lc \mc{M}_{\rho}(n) \cup \mc{M'}_{\rho}(n) \rc 
	\end{equation}
	hold, where
	$$\mc{M}_{\rho}(n) =  \bigcup_{l =0}^{2} \lfp \kk_{n} + \lc e_{n}-l \rc \cdot q_{n+1} \rfp $$
	and 
	$$\mc{M'}_{\rho}(n) =  \bigcup_{l=0}^{1} \lfp \kk_{n} + (l+1)q_{n}, \kk_{n} + q_{n+1} - l q_{n} \rfp .$$
	
\end{theorem}

Throughout this paper we use  Vinogradov's asymptotic notation: if, given two real functions $w,v: \R^{+} \to \R^{+}$, there exists a positive constant $C \;>0$ such that for every $x \in \R^{+}$  it holds that $ w(x) \le C\cdot v(x)$, then we write  $ w(x) \ll v(x)$. 
Equivalently, one may use Landau's Big-$O$ notation and write $ w(x) = O\lc v(x) \rc $. The constant $C$ is referred to as the implicit constant. If the implicit constant depends on  some parameter,  say $t$, then we index the notation as $ w(x) \ll_{t} v(x)$ (equivalently, as $w(x) = O_{t}\lc v(x) \rc$. If for two functions $w,v$ it holds that $ w(x) \ll v(x) $ and $ v(x) \ll w(x)$  for all admissible values of $x$, then we write $w(x) \asymp v(x)$. 
Two real sequences $\lc a_{n} \rc_{n\in \N}, \lc b_{n} \rc_{n\in\N}$ are called \emph{asymptotically equal} if  $ a_{n}/ b_{n} \underset{n \to +\oo}{\longrightarrow} 1$.

The paper is organized as follows.  In Section \ref{SectionAuxiliaryResults}, we first reduce the study of \eqref{eqOscillatingSequences} to that of the auxiliary sequences \eqref{eqFractionalParts}. 
In Section \ref{SectionRationalValues}, we study the case where $\rho$ is rational and establish in this case the first statement in Theorem \ref{TheorOscillatingSequence}.
In Section \ref{SectionRealValues}, we use the Ostrowski expansion \eqref{eqOstrowskiExpansion} to prove sufficient conditions for \eqref{eqFractionalParts} to be dense in  $\R$ when the root $r$ is irrational. Moreover, given parameters $\alpha$ and $\beta$ and a prescribed positive quantity $\gamma$, we provide an effective construction of the sequence $\lc e_{n} \rc_{n \ge 0}$ in the expansion \eqref{eqOstrowskiExpansion} ensuring that oscillating sequences of the form \eqref{eqOscillatingSequences} are dense in  $\R$ and satisfy $\gamma(r) = \gamma$, for some root $r$ of $F$.
In Section \ref{SectionMainResults} we use the results  from the previous sections to complete the proof of Theorem \ref{TheorOscillatingSequence} and to prove Theorem \ref{TheorIrrationalityEvaluationRationalValues}.	
In Section \ref{SectionBohrSets} we prove Theorem \ref{TheorBohrSets}.

%%\begin{figure}%%%%%%%%use black-white, grey scale or colour in cmyk palette, free-available fonts-changed to curves.
%%\begin{picture}
%%\end{picture}
%%\end{figure}
%%\begin{minipage}
%%\end{minipage}
%%
%%\begin{acknowledgement}
%%\end{acknowledgement}
%%%%%%%%%%%%%%%%%%%%%%%%%%%%%%%%%%%%%%%%%%

\paragraph{Acknowledgement.}
	The author would like to thank Faustin Adiceam for his time spent reviewing this paper, as well as for his useful comments and suggestions towards the final presentation.

\section{Some Auxiliary Results}\label{SectionAuxiliaryResults}

The goal of this section is to reduce the study of the density of sequence \eqref{eqOscillatingSequences} to that of the sequence \eqref{eqFractionalParts}. 

\begin{proposition}\label{PropOscillatingSequencesAndFractionalParts}
	Let $F: \R \to \R$ be a $1$-periodic function satisfying assumption  \eqref{eqPeriodicFunction}. Let also $g: \R^{+} \to \R^{+}$ satisfy assumption \eqref{eqGrowthRate} and let $\lc a_{k} \rc_{k \in \N}$ be a sequence of real numbers. Then, a real number $h \in \R$ is a limit point of the sequence $ \lc g(k)\cdot F\lc a_{k} \rc \rc_{k \in \N}$ if and only if there exists a root $r$ of $F$ such that $h$ lies in the closure of the set 
	$$ \lfp \e\lc \lfp a_{k} -r \rfp_{2} \rc \cdot c_{r} \cdot k^{\beta} \cdot \ldav a_{k} -r \rdav^{\gamma(r)} \rfp_{k \in \N} .$$	
\end{proposition}

To prove Proposition \ref{PropOscillatingSequencesAndFractionalParts}, we need the following lemma which  allows us to remove the error terms from the definitions of the growth rate function in \eqref{eqGrowthRate} and the periodic function in \eqref{eqPeriodicFunction}. Its proof, which is elementary, is left to the reader.

\begin{lemma}[Removing the Error Terms from Periodic Functions and Growth Rates]\label{LemRemovingErrorTerms}
	Let $\bm{f} = \lc f_{k} \rc_{k\in \N} $ be a sequence in $\R$ such that  $f_{k} \underset{k\to +\oo}{\longrightarrow} 0$. Let $g: \R^{+} \to \R^{+}$ be an increasing function such that $g(t)\underset{t\to +\oo}{\longrightarrow} +\oo$. Let also $u,v$ be real  functions such that 
	\begin{align*}
	\lim_{t\to +\oo} u(t) = 0   \quad \text{and} \quad  \lim_{x\to 0} v(x) =0 .
	\end{align*}
	Then, the sequences 
	$$ \lc g(k) \cdot f_{k} \rc_{k \in \N}, \quad \lc\lc g(k)+  u(k)\cdot g(k)\rc\cdot f_{k} \rc_{k\in \N}  $$
	and 
	$$ \lc g(k)\cdot \lc f_{k} + v\lc f_{k} \rc\cdot f_{k} \rc \rc_{k\in\N} $$
	are pairwise asymptotically equal and have therefore the same limit points.
	
\end{lemma}

We now deduce Proposition \ref{PropOscillatingSequencesAndFractionalParts}.

\begin{proof}[Proposition \ref{PropOscillatingSequencesAndFractionalParts}:]
	By assumption, the function $F$ is $1$-periodic, continuous in $\R$ and has only isolated roots in $[0,1)$. Therefore, it admits only finitely many roots in the interval $[0,1)$.
	Let $r_{0} \;< r_{1} \;<....\;< r_{m}$ be the finitely many distinct roots of $F$ in $[0,1)$. Fix $h \in \R$, where $h$ is a limit point of the sequence  $ \lc g(k)\cdot F\lc a_{k} \rc \rc_{k \in \N}$. Thus, there exists a sequence of natural numbers $\lc k_{n} \rc_{n\in\N}$  such that  $\underset{n\to+\oo}{\lim} g\lc k_{n}\rc\cdot F\lc a_{k_{n}} \rc = h$. 
	This implies that $ \underset{n\to+\oo}{\lim} F\lc a_{k_n} \rc = 0$ because $g(t) \underset{t \to +\oo}{\longrightarrow} +\oo$. By passing to a subsequence if necessary, the sequence $\lc a_{k_{n}} \rc_{n\in\N}$ converges modulo $1$ to some $r \in [0,1)$ which, by continuity, is a root of $F$. In particular, $\lfp a_{k_{n}} - r \rfp_{2} \underset{n \to +\oo}{\longrightarrow} 0$.

	Set
	$$ u(t) = {t^{\beta} - g(t) \over g(t)}  \quad \text{and} \quad v(x) = { c_{r}\cdot \e(x) \cdot |x|^{\gamma(r)} - F(r+x) \over F(r+x) } \cdotp $$
	Assumptions \eqref{eqGrowthRate} and \eqref{eqPeriodicFunction} imply that $\underset{t\to +\oo}{\lim} u(t) = 0$ and  $\underset{x\to 0}{\lim} v(x) = 0$, respectively.
	Applying Lemma \ref{LemRemovingErrorTerms} to $\bm{f} = \lc f_{k_n} \rc_{n \in \N}= \lc F\lc a_{k_n} \rc \rc_{n\in\N}$, $u$  and $v$ yields that  $h$ lies in the closure of the set 
	$$ \lfp \e\lc \lfp a_{k} -r \rfp_{2} \rc \cdot c_{r} \cdot k^{\beta} \cdot \ldav a_{k} -r \rdav^{\gamma(r)} \rfp_{k \in \N} .$$ 
	The converse follows similarly from Lemma \ref{LemRemovingErrorTerms} and assumption \eqref{eqPeriodicFunction}. The proof is complete.  \hfill $\blacksquare$
	
\end{proof}

\section{Rational Values of the Parameter $\rho$}\label{SectionRationalValues}
In this section, we study the sequence \eqref{eqFractionalParts} in the case where $ \rho \in \Q$.  To this end, we prove the following proposition which relates the quantities $\mu_{\pm}\lc \alpha,\beta,\rho\rc$ with the density in $\R$ of the sequence \eqref{eqFractionalParts}.

\begin{proposition}\label{PropRationalValuesFractionalParts}
	Let $\beta \;>0$ be a positive real number. Given an irrational number $\alpha \in \R\backslash\Q$  and a rational number $\rho$, it holds that the sequence \eqref{eqFractionalParts} is dense in $\R^{+}$ (resp. in $\R^{-}$) if and only if 
	\begin{equation}\label{eqPropRVFPAssumption1}
	\mu_{+}(\alpha,\beta,\rho) = 0 \quad \lc\text{resp.}\quad  \mu_{-}(\alpha,\beta,\rho) = 0 \rc .
	\end{equation}
\end{proposition}

\begin{proof}
	We prove the claim concerning the quantity $\mu_{+}(\alpha,\beta, \rho) $ and the density of the sequence \eqref{eqFractionalParts} in $\R^{+}$, as the claim related to $\mu_{-}(\alpha,\beta,\rho)$ and $\R^{-}$ is established in the same way. 
	
	Assume that $\mu_{+}(\alpha,\beta,\rho) = 0$. From assumption \eqref{eqPropRVFPAssumption1}, we have that, for every $n\in \N$, there exists $ m= m(n) \in \N$ such that 
	$$ 0\le \lfp m\alpha -\rho \rfp_{2} = {\e_{n} \over m^{\beta}} \;< {1\over 2} \quad \text{for some } 0\le \e_{n} \le {1\over n} .$$
	Without loss of generality, assume that 
	$$ \lfp m\alpha \rfp_{2} = \lfp \rho \rfp_{2} + {\e_{n} \over m^{\beta} }  $$
	as otherwise 
	$$ \lfp m\alpha \rfp_{2} = -1 + \lfp \rho \rfp_{2} + {\e_{n} \over m^{\beta} } , $$
	in which case we work similarly. Let us assume that $\lfp \rho \rfp_{2}= {p \over q} $ for some $p \in \Z$ and $q\in \N$ with $(p,q)=1$.  Then, for every $l\in \N_{0}$ such that 
	\begin{equation}\label{eqPropRVFP1}
	\lc lq+1 \rc\cdot {\e_{n} \over n^{\beta}} \;< {1\over 2}  ,
	\end{equation}
	it holds that $	 \lfp \lc lq+1\rc\cdot m\alpha - {p\over q} \rfp_{2} = \lc lq+1\rc\cdot {\e_{n} \over m^{\beta}} .$ 	 
	For those $l\in \N$  which satisfy inequality \eqref{eqPropRVFP1}, set
	\begin{equation}\label{eqPropRVFP2}
	Q_{\beta}(m,l) = (lq+1)^{\beta}\cdot m^{\beta}\cdot \lfp (lq+1)\cdot m\alpha - {p\over q} \rfp_{2} = \lc lq+1\rc^{1+\beta} \cdot \e_{n} ,
	\end{equation}
	where we recall that $\e_{n}$ depends  on the choice $m$. Fix $h\;>0$.
	Notice that,  for $n$ large enough, that is for $\e_{n}\le {1\over n}$ sufficiently small and $m=m(n)$ sufficiently large, the natural number 
	$$ l_{h} = \left\lceil { h^{1 \over 1+\beta} \cdot \e_{n}^{- {1 \over 1+\beta}} -1 \over q } \right\rceil $$
	satisfies inequality \eqref{eqPropRVFP1}. The quantity $Q_{\beta}(m,l_{h})$ is therefore a term in the sequence \eqref{eqFractionalParts}.
	
	The density of sequence \eqref{eqFractionalParts} follows upon noticing that 
	$$ h =  \lc \lc {h^{1\over 1+\beta}\cdot \e_{n}^{-{1 \over 1+\beta}} -1 \over q} \rc \cdot q + 1 \rc^{1+\beta} \cdot \e_{n} = Q_{\beta}\lc m(n),l_{h} \rc + O\lc h^{\beta} \cdot \e_{n}^{1\over 1+\beta} \rc $$
	and upon letting $n \to +\oo$.
	
	In the other direction, assume that  $\mu_{+}\lc \alpha,\beta, \rho\rc \;>0$. From Definition \ref{DefSignedIrrationalityEvaluation} of the quantity $\mu_{+}\lc \alpha,\beta,\rho\rc$, we have that for every $k\gg_{\alpha} 1$ such that $\lfp k\alpha - \rho \rfp_{2} \;>0 $, it holds that
	$$ k^{\beta}\cdot \lfp k\alpha - \rho \rfp_{2}  \ge C  $$
	for some positive constant $C$. Therefore, sequence \eqref{eqFractionalParts} cannot be dense in $\R^{+}$.
	
	The proof is complete. \hfill $\blacksquare$
\end{proof}

An immediate consequence of Propositon \ref{PropRationalValuesFractionalParts} is the following corollary which deals with the case when the exponent $\beta$ takes values in $(0,1)$.

\begin{corollary}
	Given $\beta\in (0,1)$, $\alpha$ irrational and $\rho$ a rational number, the sequence 
	$$ \lc k^{\beta}\cdot \lfp k\alpha - \rho\rfp_{2} \rc_{k \in \N} $$
	is dense in $\R$. Equivalently, 
	$$ \mu_{+}(\alpha,\beta,\rho) = \mu_{-}(\alpha,\beta,\rho)  = 0 .$$
\end{corollary}

\begin{proof} Let $\beta,\alpha,\rho$ as in the statement of the corollary. By the theory  of continued fractions, if ${p_{n} \over q_{n}}$ is one of the convergents of $\alpha$, then it holds that $ \ldav \alpha - {p_{n} \over q_{n}} \rdav \le {1 \over q_{n}^{2}}.$    
	Thus, we obtain easily that the finite sequence $\lc k\alpha \rc_{k=1}^{q_{n}}$ is ${2\over q_{n}}$-dense in $\T$. This implies that 
	$$\mu_{+}\lc \alpha,1, \rho \rc =  \underset{ \substack{ k\to +\oo,\\ \lfp k\alpha -r\rfp_{2} \;> 0 }}{\liminf} k\cdot \lfp k\alpha - \rho \rfp_{2}  \le 2 ,$$
	which in turn implies that $ \mu_{+}(\alpha,\beta,\rho) 	= \underset{ \substack{ k\to +\oo,\\ \lfp k\alpha -\rho\rfp_{2} \;> 0 }}{ \liminf} k^{\beta}\cdot \lfp k\alpha - \rho \rfp_{2}  = 0 .$
	We work similarly with the quantities $\mu_{-}(\alpha,1,\rho) $ and $\mu_{-}(\alpha,\beta,\rho)$. Proposition \ref{PropRationalValuesFractionalParts} now implies the result.    \hfill $\blacksquare$
	
\end{proof}

\section{Real Values of the Parameter $\rho$}\label{SectionRealValues}

The goal of this section is to use the Ostrowski expansion of a real number $\rho$ in order to
obtain sufficient conditions for the sequence \eqref{eqFractionalParts} to be dense in $\R$. This will lead us to the proof of the second statement in Theorem \ref{TheorOscillatingSequence}.

\subsection{Sufficient Conditions for Density in $\R$}

We now prove that, if $\tau_{+}\lc\alpha,\beta,\rho \rc = \tau_{-}(\alpha,\beta,\rho) =0$, then the sequence \eqref{eqFractionalParts} is dense in $\R$. Moreover, in the case where $\mu_{+}(\alpha,\beta) = \mu_{-}(\alpha,\beta) = 0$, the proof provides an effective way to construct the coefficients in the Ostrowski expansion \eqref{eqOstrowskiExpansion}, and thus the parameter $\rho$, for the sequence \eqref{eqFractionalParts} to enjoy the density property.

\begin{proposition}\label{PropRealValuesSufficientConditions}
	Given $\beta \;>0 $ and $\alpha, \rho \in \R\backslash\Q$, consider the Ostrowski expansion of $\rho$ as defined in \eqref{eqOstrowskiExpansion}.  If 
	$$ \tau_{+}(\alpha,\beta,\rho) = 0 \quad \lc \text{resp.}\quad \tau_{-}(\alpha,\beta,\rho) = 0  \rc, $$
	then the sequence \eqref{eqFractionalParts} is dense in $\R^{+}$ (resp. in $\R^{-}$).
	
\end{proposition}

Before we continue, recall some facts which will be used extensively in the forthcoming proofs. For every $x \in \R$, it holds that $ -\ldav x \rdav \le \lfp x \rfp_{2} \le \ldav x \rdav$. Also, given an irrational $\alpha$ and the Ostrowski expansion \eqref{eqOstrowskiExpansion} of a real number $\rho$, we have that, for every $n \in \N$ 
\begin{equation}\label{eqOstrowskiResidue}
\lav \sum_{j=n+1}^{+\oo} e_{j}(\rho) \cdot \lfp q_{j}\alpha \rfp_{2} \rav \le \ldav q_{n} \alpha \rdav .
\end{equation}
Indeed, by the definition of the continued fraction expansion of a real number $\alpha$ (see \cite[Section 3]{BeresnevichSumsOfReciprocals}), we have that $ a_{1}\cdot \lfp \alpha \rfp -1 = \lfp q_{1} \alpha \rfp_{2} $, $ \lfp \alpha \rfp + a_{2}\cdot \lfp q_{1} \alpha \rfp_{2} = \lfp q_{2} \alpha \rfp_{2} $ and, for every $n \ge 1$, it holds that $\lfp q_{n}\alpha \rfp_{2} + a_{n+2}\cdot \lfp q_{n+1}\alpha \rfp_{2} = \lfp q_{n+2} \alpha \rfp_{2} $. This implies that 
\begin{equation}\label{eqContinuedFractionsIntervals}
\lfp q_{n} \alpha \rfp_{2} = -  \sum_{j\ge 1} a_{n+2i} \cdot \lfp q_{n +2i-1} \cdot \alpha \rfp_{2} \quad\text{for every } n\ge1.
\end{equation}
In turn, this implies that 
$$ \lav \sum_{j=n+1}^{+\oo} e_{j}(\rho)\cdot \lfp q_{j}\alpha \rfp_{2} \rav \le \lav \sum_{i=1}^{+\oo} a_{2i+n}\cdot \lfp q_{2i-1+n}\alpha\rfp_{2} \rav  = \ldav q_{n} \alpha \rdav ,    $$
whence the claim.

\begin{proof}[Proposition \ref{PropRealValuesSufficientConditions}:]
	We prove the result regarding the quantity $\tau_{+}(\alpha,\beta,\rho)$ and the density of \eqref{eqFractionalParts} in $\R^{+}$. The other case follows in the same way. To this end, assume that $\tau_{+}(\alpha,\beta,\rho) =0$. Given $j\ge0$, set $e_{j} = e_{j}(\rho)$.
	
	\paragraph{Case 1:} Assume that 
	\begin{equation}\label{eqPropRVSCAssumption1}
	\liminf_{ \substack{ j \to +\oo }} \max\lfp 1, e_{2j}^{\beta} \rfp \cdot q_{2j}^{\beta} \cdot \lfp q_{2j}\alpha \rfp_{2}  = 0 .
	\end{equation}

	Fix $n \in \N$. There exists $m=m(n) \in 2\N$  such that
	\begin{equation}\label{eqPropRVSC1}
	\e_{n} := q_{m}^{\beta}\cdot \lfp q_{m} \alpha \rfp_{2} \le {1\over n} \quad\quad \text{and}\quad \quad  e_{m}^{\beta} \cdot \e_{n} \le {1\over n} \cdotp
	\end{equation}
	
	Since $m \in 2\N$, one has that $ \lfp q_{m} \alpha \rfp_{2} = \ldav q_{m} \alpha \rdav$ and thus  inequality \eqref{eqOstrowskiResidue}  yields that \newline
	$$ \lav \sum_{j=m+1}^{+\oo} e_{j}\cdot \lfp q_{j}\alpha \rfp_{2} \rav \le \lfp q_{m}\alpha \rfp_{2} .$$
	Set 
	\begin{equation}\label{eqPropRVSC2}
	\h = q_{m}^{\beta}\cdot \sum_{j=m+1}^{+\oo} e_{j}\cdot \lfp q_{j}\alpha \rfp_{2} ,
	\end{equation}
	so that $|\h| \le \e_{n}$. Given $l\in \N$ such that  $l\e_{n} -\h \;< {1\over 2 }$, set
	\begin{equation*}
	\begin{split}
	Q_{\beta}\lc m,l \rc &= \lc  \sum_{j=0}^{m-1} e_{j}\cdot q_{j} + e_{m}\cdot q_{m} +l\cdot q_{m} \rc^{\beta}\cdot  \\  &\qquad\qquad \lfp \lc \sum_{j=0}^{m-1} e_{j}\cdot q_{j} 	 + e_{m}\cdot q_{m} + l\cdot q_{m} \rc \cdot \alpha - \rho \rfp_{2} \\
	&\underset{\eqref{eqOstrowskiExpansion},\eqref{eqPropRVSC1},\eqref{eqPropRVSC2} }{=}  \lc { \sum_{j=0}^{m-1} e_{j}\cdot q_{j} \over q_{m} } + e_{m} +l \rc^{\beta}\cdot \lc l\cdot \e_{n} - \h \rc .
	\end{split}
	\end{equation*}
	It easily follows from the Ostrowski expansion of a natural number (see \cite[Lemma 3.1]{BeresnevichSumsOfReciprocals})  that
	$$ { \sum_{j=0}^{m-1} e_{j}\cdot q_{j} \over q_{m} } \le 1 .$$
	In turn, from inequalities \eqref{eqPropRVSC1} one infers that $\lav Q_{\beta}\lc m,0 \rc \rav \ll 1/n$. 
	
	Fix $h\;>0$. Note that for $l' = 2 \cdot \lf h^{1 \over \beta +1}\cdot \e_{n}^{-{1 \over \beta +1}} \rf $ and  for $n$ large enough, it holds that $l'\e_{n} - \h \le {1\over 4}$ thanks to relations \eqref{eqPropRVSC1} and \eqref{eqPropRVSC2}. Thus, for every $l \in \ldb 0, l'\rdb$, the quantity $Q_{\beta}(m,l)$ is a term of the sequence \eqref{eqFractionalParts}.  Moreover, it holds that  $ Q_{\beta}(m,l') \;> h $. We can then use the relations in \eqref{eqPropRVSC1} in order to prove that, for every  $l \in\ldb 0, l'\rdb$, 
	\begin{equation*}  
	\begin{split}
	  \lav Q_{\beta}\lc m(n), l+1 \rc - Q_{\beta}\lc m(n), l \rc \rav 
	 \ll& \quad  h^{1\over 1+\beta} \cdot \lc {1\over n} \rc^{\beta \over \beta +1}  + h^{\beta \over \beta +1} \cdot \lc {1\over n} \rc^{1\over \beta +1} + {1\over n}  \\
	 =:& \quad \h\lc n, h \rc.
	 \end{split}
	 \end{equation*}
	
	Since $\lav Q_{\beta}\lc m(n),0 \rc \rav \ll 1/n$ and $ Q_{\beta}\lc m(n), l' \rc \ge h$, the last inequality yields that the terms $ \lfp Q_{\beta}\lc m(n),l \rc \rfp_{l \in \ldb 0, l'\rdb}$ partition the interval $\lb 0, h \rb$ into subintervals with length at most $O\lc \h(n,h)\rc$. Since $\h(n,h) \to 0$ when $n \to+\oo$ and the choice of $h\;>0$ was arbitrary, one infers that the sequence \eqref{eqFractionalParts} is dense in $\R^{+}$.  
	
	\paragraph{Case 2:} Let us assume that 
	\begin{equation}\label{eqPropRVSCAssumption2}
	\liminf_{ \substack{ j \to +\oo }} \max\lfp 1, \lc a_{2j+1} - e_{2j}\rc^{\beta +1 \over 2} \rfp \cdot q_{2j}^{\beta} \cdot \lfp q_{2j}\alpha \rfp_{2}  = 0 .
	\end{equation}
	Without loss of generality, assume that $ \beta \ge 1$ as otherwise  assumption \eqref{eqPropRVSCAssumption1} dealt with in Case 1 holds. 
	We will follow arguments similar to the first case. 	Fix $n\in\N$.	Then, there exists $m= m(n) \in 2\N$  such that 
	\begin{equation}\label{eqPropRVSC3}
	\e_{n} := q_{m}^{\beta}\cdot \lfp q_{m}\alpha \rfp_{2}\le {1\over n} \quad \text{and}\quad  
	\lc a_{m+1} - e_{m} \rc^{\beta +1 \over 2}\cdot \e_{n}  \le {1\over n} \cdotp
	\end{equation}
	Define  $\h$ as in \eqref{eqPropRVSC1} in such a way that  $|\h|\le \e_{n}$. Also, set $\h'  = q_{m}^{\beta}\cdot \lfp q_{m+1}\alpha \rfp_{2}$, wherefrom it follows that $|\h'| \le \e_{n}$. 
	
	Given $l\ge 1$  such that   $l\cdot \e_{n} + \lc a_{m+1} - e_{m} \rc\cdot\e_{n} 	-\h' - \h \;< {1 \over 2} $, let
	\begin{equation*}
	\begin{split}
	P_{\beta}(m,l) =&  \lc \sum_{j=0}^{m-1} e_{j}\cdot q_{j}  + l \cdot q_{m} - q_{m-1} \rc^{\beta}\cdot \\ &\qquad \qquad \lfp \lc \sum_{j=0}^{m-1}e_{j}\cdot q_{j} + l\cdot q_{m} - q_{m-1} \rc\cdot \alpha - \rho \rfp_{2}\\
	\underset{\eqref{eqOstrowskiExpansion}}{=}&   \lc \sum_{j=0}^{m-1}   e_{j}\cdot q_{j}  + l \cdot q_{m} - q_{m-1} \rc^{\beta} \cdot  \\ &  \qquad
	\lfp   \lc l+ a_{m+1} -e_{m}\rc\cdot \lfp q_{m}\alpha \rfp_{2}  - \lfp q_{m+1}\alpha \rfp_{2}   - \sum_{j=m+1}^{+\oo} e_{j}\cdot\lfp q_{j} \alpha \rfp_{2}    \rfp_{2}  \\
	 \underset{\eqref{eqPropRVSC2},\eqref{eqPropRVSC3}}{=} & \lc { \sum_{j=0}^{m-1} e_{j}\cdot q_{j}  - q_{m-1} \over q_{m} } + l \rc^{\beta} \cdot \lc  l\e_{n}  + \lc a_{m+1} - e_{m} \rc \cdot \e_{n} - \h' - \h \rc .
	\end{split}
	\end{equation*}
	As in the previous case, the Ostrowski expansion of a natural number yields
	$$   \lav { \sum_{j=0}^{m-1} e_{j}\cdot q_{j}  - q_{m-1} \over q_{m} } \rav \le 1 .$$
	In turn, from inequalities \eqref{eqPropRVSC3} one infers that $\lav P_{\beta}(m,0) \rav \ll 1/n$. 
	
	Fix $h\;>0$. For $l' = 4 \cdot \lf h^{1 \over \beta +1} \cdot \e_{n}^{-{1 \over \beta +1}} \rf$ and $n$ large enough, inequalities \eqref{eqPropRVSC3} and \eqref{eqPropRVSC2}  imply that $l'\cdot \e_{n} + \lc a_{n+1} - e_{n} \rc\cdot\e_{n} 	-\h' - \h \le {1 \over 4} $. Thus, given $l\in\ldb 0,l'\rdb$, the quantity $P_{\beta}(m,l)$ is a term in the sequence \eqref{eqFractionalParts}. Moreover, it holds that $ P_{\beta}\lc m,l' \rc  \;> h $. We can then use the relations in \eqref{eqPropRVSC3} in order to prove that, for every $l \in \ldb 0, l'\rdb$, 
	$$ \lav P_{\beta}\lc m(n),l+1 \rc - P_{\beta}\lc m(n),l \rc \rav \ll h^{\beta \over \beta +1} \cdot \lc {1\over n} \rc^{1\over \beta+1} + h^{\beta-1 \over \beta +1} \cdot \lc {1\over n} \rc^{2\over \beta +1}   \quad =: \h(n,h).$$
	
	Since $\lav P_{\beta}\lc m(n), 0 \rc \rav \ll 1/n$  and $P_{\beta}\lc m(n), l'\rc \;> h$, the last inequality yields that the terms $ \lfp P_{\beta}\lc m(n),l \rc \rfp_{l \in \ldb 0, l'\rdb}$ partition the interval $\lb 0, h \rb$ into subintervals with length at most $O\lc \h(n,h)\rc$.
	Since $\h(n,h) \to 0$ when $n \to+\oo$ and the choice of $h\;>0$ was arbitrary, one infers that the sequence \eqref{eqFractionalParts} is dense in $\R^{+}$.  
	
	The proof is complete.  \hfill $\blacksquare$
\end{proof}

The following corollary is a straightforward consequence of Proposition \ref{PropRealValuesSufficientConditions}.

\begin{corollary}\label{CorSmallValuesOfBeta}
	Let $\beta \in (0,1)$ be a real number. Let also $\alpha \in \R\backslash \Q$ be an irrational and let $\rho$ be a real number. Then, the sequence 
	$$ \lc k^{\beta}\cdot \lfp k\alpha -\rho \rfp_{2} \rc_{k \in \N} $$ 	
	is dense in $\R$. 	
\end{corollary}

\begin{proof}
	Let $\beta,\alpha,\rho$ be as in the statement.  Assume that $\lc e_{j}\rc_{j\ge 0}$ is the sequence of the digits in the Ostrowski	 expansion of $\rho$. From the theory of continued fractions, for every $n \in\N$, it holds that $\ldav q_{n}\alpha \rdav \le {1\over a_{n+1}q_{n}}$. Consequently,
	\begin{align*}
	\liminf_{j\to +\oo} q_{2j}^{\beta}\cdot \max\lfp e_{2j}^{\beta} , 1 \rfp \cdot \lfp q_{2j}\alpha \rfp_{2} 
	&\le  \liminf_{j\to +\oo}  \max\lfp {  e_{2j}^{\beta} \cdot q_{2j}^{\beta}  \over a_{2j+1}\cdot q_{2j}}, {q_{2j}^{\beta}  \over a_{2j+1}\cdot q_{2j}} \rfp \\
	&\le  \liminf_{j \to +\oo} {1 \over q_{2j}^{1-\beta}} = 0 .
	\end{align*}
	Similarly, we can show that  $ \underset{j\to +\oo}{\liminf} - q_{2j+1}^{\beta}\cdot \max\lfp 1, e_{2j+1}^{\beta} \rfp \cdot \lfp q_{2j+1}\alpha \rfp_{2} = 0 .$
	Therefore, Proposition \ref{PropRealValuesSufficientConditions} implies that the sequence \eqref{eqFractionalParts} is dense in $\R$.  The proof is complete.  	\hfill $\blacksquare$
\end{proof}

Proposition \ref{PropOscillatingSequencesAndFractionalParts} and Corollary \ref{CorSmallValuesOfBeta} immediately imply Corollary \ref{CorOscillatingSequenceSmallValuesOfBeta}.

\subsection{Effective Construction of the Parameter $\rho$} 

The sufficient condition in the statement of Proposition \ref{PropRealValuesSufficientConditions} is not necessary. Indeed, in the following proposition we construct real numbers $\rho \in \R$ such that the sequence \eqref{eqFractionalParts} is dense in $\R$ but with $\tau_{\pm}(\alpha,\beta,\rho) = +\oo$.

\begin{proposition}
	Let $\beta \;>0$ be a positive number and $\alpha$ be an irrational such that $\mu_{+}(\alpha,\beta)$ and $\mu_{-}(\alpha,\beta)$ equal either zero or infinity. Then, there exists an effectively constructible sequence of digits $ \lc e_{j} \rc_{j\ge 0}$ in the Ostrowski expansion \eqref{eqOstrowskiExpansion} of the real number $\rho$ such that the sequence $\lc w_{k} \rc_{k \in\N}$ defined in \eqref{eqFractionalParts} is dense in $\R$.	Moreover, there exist uncountably many such numbers $\rho$.
\end{proposition}

\begin{proof}
	We split the proof into three cases depending on the values of $\mu_{\pm}(\alpha,\beta)$.
	
	\paragraph{Case 1:} Assume that $\mu_{+}(\alpha,\beta)= \mu_{-}(\alpha,\beta) = 0$. Then, the result follows easily from Proposition \ref{PropRealValuesSufficientConditions}. For instance, for every $j\ge 0$, we can choose $e_{j} \in \lfp 0,1\rfp$ so that the resulting sequence is dense in $\R$.
	
	\paragraph{Case 2:} Assume that $\mu_{+}(\alpha,\beta) = \mu_{-}(\alpha,\beta) = +\oo$; that is, that
	\begin{equation}\label{eqPropRVECAssumption1} 
	\liminf_{n \to +\oo}	q_{2n}^{\beta}\cdot \lfp q_{2n}\alpha \rfp_{2} = \liminf_{n \to +\oo} -	q_{2n+1}^{\beta}\cdot \lfp q_{2n+1}\alpha \rfp_{2} = +\oo .
	\end{equation}
	
	Fix a sequence $\bm{b}= \lc b_{j} \rc_{j\in \N}$ of real numbers which is dense in $\R$.
	The goal is to define the sequence $\lc e_{j} \rc_{j\ge 0}$ recursively in such a way that 
	\begin{equation}\label{eqPropRVEC1}
	\lav b_{j} - w_{\kk_{m(j)}} \rav \underset{j \to +\oo}{ \longrightarrow} 0 ,
	\end{equation}
	where $ \lc \kk_{m(j)} \rc_{j \in \N}$ is  a proper subsequence of the sequence \eqref{eqOstrowskiSequence} defined in the course of the proof below. Relation \eqref{eqPropRVEC1} then yields the density of sequence \eqref{eqFractionalParts}.
	
	Choose $\e_{0} \in \ldb 0, a_{1}-1 \rdb$ arbitrary and fix $ j\in \N$. If $j= 1$, then, without loss of generality, assume that $b_{1} \;>0$. From equation \eqref{eqPropRVECAssumption1}, there exists $m(1) \in 2\N$ such that $  q_{m(1)}^{\beta}\cdot \lfp q_{m(1)}\cdot\alpha \rfp_{2} \ge 5 b_{1} $. Given $n \in \ldb 1, m(1) -1 \rdb$, set $e_{n} =0$ and choose $e_{m(1)} \in \ldb 1, a_{m(1)+1} \rdb$ arbitrary. Fix $l(1) \in \N$ large enough. From equation \eqref{eqContinuedFractionsIntervals} and the choice of $m(1)$, for every $n \in \ldb m(1) +1 ,m(1) + l(1) \rdb$, the digits $e_{n} \in \ldb 0 , a_{n+1} \rdb$ can be chosen in such a way that 
	$$ \lav b_{1} + \kk_{m(1)}^{\beta}\cdot \sum_{n= m(1)+1}^{m(1)+l(1)} e_{n}\cdot \lfp q_{n}\alpha \rfp_{2} \rav \le {1\over 2} \cdotp $$

	If $j\ge 2$, then assume that the numbers  $m(j-1),l(j-1) \in \N$ have been chosen in such a way  that, for every $n \in\ldb 1, m(j-1) +l(j-1) \rdb$, the digits $e_{n} \in \ldb 0 , a_{n+1} \rdb$ are such that for every $j' \in \ldb 1, j-1 \rdb$, 
	$$ \lav b_{j'} + \kk_{m(j')}^{\beta}\cdot \sum_{n= m(j')+1}^{m(j')+ l(j')} e_{n}\cdot \lfp q_{n} \alpha \rfp_{2} \rav \le {1 \over 2j'} $$
	and 
	$$ \kk_{m(j'-1)}^{\beta}\cdot \ldav q_{m(j')-2} \cdot \alpha \rdav \le {1 \over 	2(j'-1)} \cdotp $$

	Without loss of generality, assume that $b_{j} \ge 0$. From equation \eqref{eqPropRVECAssumption1}, there exists $m_{j} \in 2\N$  such that $m(j) \ge m(j-1) + l(j-1) +1$,
	\begin{equation}\label{eqPropRVEC2}
	q_{m(j)}^{\beta}\cdot \lfp q_{m(j)}\cdot \alpha \rfp_{2} \ge 5b_{j} \quad \text{and}\quad  \kk_{m(j-1)}^{\beta}\cdot \ldav q_{m(j)-2} \cdot \alpha \rdav \le {1 \over 	2(j-1)} ,
	\end{equation}
	where the last inequality holds if $m(j) $ is chosen large enough.	
	Here, the constant $5$ in the left inequality ensures that the choice of the digits $e_{n}$ in the next step of the induction satisfies the properties of the Ostrowski expansion (as given in  relation \eqref{eqOstrowskiExpansion}).
	
	Given $n \in \ldb m(j-1) + l(j-1)+1 , m(j) -1 \rdb$, set $e_{n} =0 $ and choose $e_{m(j)} \in \ldb 1, a_{m(j)+1} \rdb$ arbitrary. 
	Fix  $l(j) \in \N$ large enough. From equation \eqref{eqContinuedFractionsIntervals} and the left inequality of \eqref{eqPropRVEC2}, for every $n \in \ldb m(j)+1, m(j) + l(j) \rdb$, the digits $e_{n} \in \ldb 0, a_{n+1} \rdb$ can be chosen in such a way  that 
	$$ \lav  b_{j} +  \kk_{m(j)}^{\beta}\cdot \sum_{n= m(j)+1}^{m(j)+l(j)} e_{n}\cdot \lfp q_{n} \alpha \rfp_{2} \rav \le {1 \over 2j} \cdotp $$	
	In the case where $b_{j} \;< 0$, one works in a similar way by choosing $m(j) \in 2\N+1$ large enough. Therefore, we have defined the sequence $\lc e_{n} \rc_{n\in \N}$ and can thus set 
	$\rho = e_{0}\cdot \lfp \alpha \rfp +  \sum_{n=1}^{+\oo} e_{n} \cdot \lfp q_{n} \alpha \rfp_{2}$. It is not hard to check that for every $j \in \N$, it holds that 
	$$ \lav b_{j} - w_{\kk_{m(j)}} \rav \le {1\over j} .$$
	The claim is thus proved.
	
	\paragraph{Case 3:} Assume that one of the quantities $\mu_{\pm}(\alpha,\beta)$ equals zero and the other one equals infinity. For instance, assume that $\mu_{+}(\alpha,\beta) = +\oo$ and $\mu_{-}(\alpha,\beta) =0$.  Fix a sequence $\bm{b} =\lc b_{j} \rc_{j\in \N}$ of real numbers which is dense in $\R^{+}$. We follow the steps in the proof of the second case but this time we choose $m(j)\in 2\N$ large enough so that $q_{m(j)-1}^{\beta} \cdot \ldav q_{m(j)-1}\alpha \rdav \underset{ j \to +\oo}{\longrightarrow} 0 $ and 
	$e_{m(j)-1} \in \lfp 0, 1 \rfp$. The density in $\R^{+}$ follows from the arguments presented in the second case, and the density in $\R^{-}$ follows from Proposition \ref{PropRealValuesSufficientConditions}. When $\mu_{+}(\alpha,\beta) =0$ and $\mu_{-}(\alpha,\beta) = +\oo$, one works similarly.

	The arguments in all three cases imply easily the construction of uncountably many such numbers $\rho$.
	The proof is complete.  \hfill $\blacksquare$
	
\end{proof}

\begin{remark}
	Given $\beta\;>0$ and an irrational $\alpha$, it can be shown that there exist (uncountably many) real numbers $\rho$ such that the sequence \eqref{eqFractionalParts} is dense in $\R$. However, if at least one of the quantities $\mu_{\pm}(\alpha,\beta)$ is positive and finite, then, it is not clear to the author how one can effectively construct the digits $\lc e_{j} \rc_{j\in \N}$ of the Ostrowski expansion \eqref{eqOstrowskiExpansion} of the real $\rho$. Note that given $\alpha \in \R\backslash \Q$, there exists at most one real number $\beta_{+}\;>0$ (resp. $\beta_{-} \;>0$) such that $\mu_{+}\lc\alpha,\beta_{+}\rc \in (0,+\oo)$ (resp. $\mu_{-}\lc \alpha,\beta_{-} \rc \in (0,+\oo)$). 
\end{remark}

\section{Proof of Theorems \ref{TheorOscillatingSequence} and \ref{TheorIrrationalityEvaluationRationalValues} }\label{SectionMainResults}

We are now ready to prove Theorem \ref{TheorOscillatingSequence} and Corollary \ref{CorTrigonometricOscillatingSequence}.

\begin{proof}[Theorem \ref{TheorOscillatingSequence}]
	As far as the first part of the theorem is concerned, assume that the sequence $\lc y_{k} \rc_{k\in \N}$ defined in \eqref{eqOscillatingSequences} is dense in $\R^{+}$. Then, there exists an increasing sequence of natural numbers $\lc k_{n} \rc_{n \in \N}$ such that, for every $n \in \N$, $ F\lc k_{n}\alpha \rc \;>0 $ and  $ g\lc k_{n} \rc\cdot F\lc k_{n}\alpha \rc \le {1 \over n}$. By passing to a subsequnce if necessary, the continuity of $F$ yields that $ \ldav k_{n}\alpha - r\rdav \underset{ n \to +\oo}{ \longrightarrow} 0$ for some root $r$ of $F$, and the claim follows. Work similarly in the case where $\lc y_{k} \rc_{k \in \N}$ is dense in $\R^{-}$. In the special case where the root $r$ is rational, an immediate application of Propositions \ref{PropRationalValuesFractionalParts} and \ref{PropOscillatingSequencesAndFractionalParts} implies the claim.
	
	The second part of the theorem follows from Propositions \ref{PropRealValuesSufficientConditions} and \ref{PropOscillatingSequencesAndFractionalParts}.   \hfill $\blacksquare$
\end{proof}

\begin{proof}[Corollary \ref{CorTrigonometricOscillatingSequence}] The function $F(x) = \sin\lc 2\pi \cdot x \rc$ is easily seen to satisfy assumption \eqref{eqPeriodicFunction}. Moreover, all its roots are rationals. The result now follows upon noticing that, given $\alpha$ irrational, $\beta \;>0$ and $\rho$ a rational number, $\mu_{+}\lc \alpha,\beta,\rho \rc = 0$ (resp.  $\mu_{-}\lc\alpha,\beta, \rho\rc = 0$) implies $\mu_{+}(\alpha,\beta) = 0$ (resp.  $\mu_{-}(\alpha,\beta) =0$). This claim follows from Theorem \ref{TheorIrrationalityEvaluationRationalValues}.
	The proof of the corollary is complete.  	 \hfill $\blacksquare$
\end{proof}

We now prove Thoerem \ref{TheorIrrationalityEvaluationRationalValues}.

\begin{proof}[Theorem \ref{TheorIrrationalityEvaluationRationalValues}]
	We will prove only the case dealing with the quantities 
	$$ \lim_{ \substack{j \to +\oo,\\ q|q_{j}, j \in 2\N} } q_{j}^{\beta}\cdot\lfp q_{j}\alpha \rfp_{2} $$
	and $\mu_{+}\lc\alpha,\beta,\rho\rc$. The other case is similar.
	
	\paragraph{}
	$\Rightarrow:$ Fix $\e'\;>0$ and let $\lc q_{n} \rc_{n\in\N}$ be the sequence of denominators of convergents of $\alpha$. Assume that 
	$$ \lim_{ \substack{j \to +\oo,\\ q|q_{j}, j \in 2\N} } q_{j}^{\beta}\cdot\lfp q_{j}\alpha \rfp_{2} = 0 .$$
	Then, there exists $n \in 2\N$  such that $q|q_{n}$ and, for this $q_{n}$, it holds that 
	\begin{equation}\label{eqTheorIERV1}
	0\;< q_{n}^{\beta} \cdot \lfp q_{n}\alpha\rfp_{2} := \e \le \e' .
	\end{equation}
	Since $n$ is even, the theory of continued fractions implies that $\lfp q_{n}\alpha \rfp_{2} \;>0$. We obtain immediately that 
	\begin{equation}\label{eqTheorIERV2}
	\alpha = { p_{n} \over q_{n}} + {\e \over q_{n}\cdot q_{n}^{\beta}} 
	\end{equation}
	for some $p_{n} \in \Z$ with $(p_{n} ,q_{n})=1$. Write $q_{n} = q\cdot q_{n}'$ for some $q_{n}'\in \N$ and choose $p_{n}'\in \lfp 1,...,q-1 \rfp$ such that $ p_{n}'\cdot p_{n} \equiv p \pmod q $. Then,
	\begin{align*}
	\lc p_{n}'\cdot q_{n}'\rc^{\beta}\cdot \lfp p_{n}'q_{n}'\cdot \alpha - \rho \rfp_{2} &\underset{\eqref{eqTheorIERV2} }{=} \lc p_{n}'\cdot q_{n}' \rc^{\beta} \cdot \lfp p'_{n}\cdot q'_{n}\cdot \lc {p_{n} \over q_{n}} + {\e \over q_{n}\cdot q_{n}^{\beta}} \rc - {p \over q} \rfp_{2} \\
	& \underset{ \lc \substack{p'_{n}\cdot p_{n} \equiv p  \\ \pmod q} \rc}{=}  \lc {p'_{n}\cdot q'_{n} \over q_{n} }\rc^{1+\beta} \cdot \e  \underset{ \lc p \;< q \rc}{\le} \e .
	\end{align*}
	This implies that $\mu_{+}\lc \alpha,\beta,{p\over q} \rc \le \e'$. Therefore, $\mu_{+}\lc \alpha,\beta, {p\over q} \rc = 0$ as $\e'$ is chosen arbitrary.
	
	\paragraph{}
	$\Leftarrow:$ Assume that $\mu_{+}\lc \alpha,\beta,{p \over q} \rc = 0$. Without loss of genrality, assume that $ p/q \in [0,1)$.  We prove first the case $q\ge 2$.  Fix
	\begin{equation}\label{eqTheorIERV3}
	0 \;< \e_{0} \;< { 1\over 2\cdot q^{2+\beta}} \cdotp
	\end{equation}
	By assumption, there exists $k \in \N$ such that 
	$$ 0 \;< k^{\beta}\cdot \lfp k\alpha -\theta \rfp_{2} \le \e_{0} .$$
	Set 
	$$  \e = k^{\beta}\cdot \lfp k\alpha - \theta \rfp_{2}.$$
	Then, 
	\begin{equation}\label{eqTheorIERV4}
	\lfp k\alpha \rfp  = \theta  + {\e \over k^{\beta}} = {p\over q} + { \e \over k^{\beta}} \cdotp
	\end{equation}
	From inequality \eqref{eqTheorIERV3}, one obtains that ${q\e / k^{\beta}} \in \lb -{1\over 2}, {1\over 2} \rc$. Therefore, equation \eqref{eqTheorIERV4} yields
	$$ \lfp qk\cdot \alpha \rfp_{2} = {q\e \over k^{\beta}}  \cdotp $$
	
	Let $n\in \N$ be such that $q_{n} \le qk \;< q_{n+1}$. Then, it holds that $\ldav q_{n}\alpha \rdav \le {q\e \over k^{\beta}}$. Also, 
	$$ \alpha = {p_{n} \over q_{n}} + (-1)^{n}\cdot  {\e' \over q_{n}^{1+\beta}} $$
	for some $\e' \;>0$  and $p_{n} \in \N$  with $\lc p_{n},q_{n} \rc =1$.  Indeed, from the theory of continued fractions, we have that $\alpha = {p_{n} \over q_{n}} + (-1)^{n}\cdot{ \h\over q_{n}^{2}} $ for some  $\h\;>0$. Setting $\e'= q_{n}^{\beta-1} \cdot \h$ implies that
	\begin{equation}\label{eqTheorIERV5}
	{\e' \over q_{n}^{\beta}} \le {q\e \over k^{\beta}} \cdotp
	\end{equation}
	
	Let us prove that $q|q_{n}$ and $n \in 2\N$. Choosing $\e_{0}$ sufficiently small yields that $q_{n} \ge q$. Therefore, without loss of generality, assume for the rest of the proof that $q_{n} \ge q$.
	
	Assume that $q \not| q_{n}$. Then, for every $j \in \N$, 
	$$ \ldav {j\over q_{n}} - {p\over q} \rdav \ge {1 \over qq_{n}} $$
	since ${j\over q_{n}} \neq {p\over q} \pmod 1$ for all $j \in \Z$. Thus, in order  for the relations
	\begin{equation}\label{eqTheorIERV6}
	k\cdot \lc {p_{n}\over q_{n}} + {\e' \over q_{n}^{1+\beta}} \rc = \theta +{\e \over k^{\beta} }  \pmod 1 \quad\text{or} \quad  k\cdot \lc {p_{n}\over q_{n}} - {\e' \over q_{n}^{1+\beta}} \rc = \theta +{\e \over k^{\beta}}  \pmod 1 
	\end{equation}
	to hold, it is necessary  that 
	\begin{equation}\label{eqTheorIERV7}
	k\cdot { \e' \over q_{n}^{1+\beta}}  \ge {1 \over 2  q\cdot q_{n}} \cdotp
	\end{equation}
	However, 
	$$ k\cdot { \e' \over q_{n}^{1+\beta}}  \quad \underset{\eqref{eqTheorIERV5}}{\le} \quad kq\cdot {\e \over q_{n}\cdot k^{\beta}} \quad \underset{\beta\ge 1}{\le} \quad q\cdot {\e \over q_{n}} \quad \underset{\e \;< \e_{0}}{\;<} \quad {1 \over 2 q_{n}\cdot q^{1+\beta}} \quad \le \quad {1 \over 2q_{n}\cdot q} \cdotp $$
	This contradiction establishes that $q|q_{n}$. Set now $\delta = (-1)^{n}\cdot \e'$ and write 
	$ \alpha = { p_{n} \over q_{n}} + {\delta \over q_{n}^{1+\beta}} .$ If $n$ is odd, then relation \eqref{eqTheorIERV6} holds only if inequality \eqref{eqTheorIERV7} is true. This leads again to a contradiction, establishing this way that $n$ is even and, in particular, that
	$$ \alpha = {p_{n} \over q_{n}}  + { \e' \over q_{n}^{1+\beta}} \quad \text{ with } \quad	k\cdot {\e' \over q_{n}^{1+\beta} }  = {\e \over k^{\beta}} .$$
	Finally, one has that
	\begin{align*}
	 \liminf_{ \substack{j \to +\oo,\\ q|q_{j}} } q_{2j}^{\beta}\cdot\lfp q_{2j}\alpha \rfp_{2} \quad \le& \quad q_{n}^{\beta}\cdot \lfp q_{n}\alpha \rfp_{2} \quad \underset{\eqref{eqTheorIERV5}}{ \le } \quad q_{n}^{\beta}\cdot {q\e \over k^{\beta}} \\ 
	 \underset{ q_{n} \le qk }{\le }& \quad q^{1+\beta}\cdot \e  \quad \le \quad q^{1+\beta}\cdot \e_{0} .  
	 \end{align*}
	By letting $\e_{0} \to 0$, one obtains that
	$$ \lim_{ \substack{j \to +\oo,\\ q|q_{j} } } q_{2j}^{\beta}\cdot\lfp q_{2j}\alpha \rfp_{2} = 0 .$$ 
	
	It remains to establish the case $q=1$; that is, when $ \theta \in \Z$. Assume that $\mu_{+}(\alpha,\beta) = 0$. The goal is to prove that 
	$$ \lim_{ \substack{j \to +\oo } } q_{2j}^{\beta}\cdot\lfp q_{2j}\alpha \rfp_{2} = 0 .$$ 
	The following lemma immediately implies the claim.
	
	\begin{lemma}\label{LemIrrationalityEvaluationRationalValues}
		Given $\alpha \in \R\backslash\Q$ and $\beta \ge 1$, the relations 
		$$ \mu_{+}(\alpha,\beta)  = \liminf_{ \substack{ j\to +\oo }} q_{2j}^{\beta}\cdot \lfp q_{2j}\alpha \rfp_{2}  \quad \text{and}\quad \mu_{-}(\alpha,\beta) =  \liminf_{ \substack{ j\to +\oo }} -q_{2j-1}^{\beta}\cdot \lfp q_{2j-1}\alpha \rfp_{2}  $$
		hold.	
	\end{lemma}
	
	\begin{proof}
		We prove the first relation as the second one follows in the same way. To this end, fix an even integer $n$. It is enough to prove that 
		$$ \lc q_{n}  + l \cdot q_{n+1}  \rc^{\beta} \cdot \lfp \lc q_{n} + l\cdot q_{n+1} \rc\cdot \alpha \rfp_{2} \ge q_{n}^{\beta}\cdot  \lfp q_{n} \alpha \rfp_{2} $$ 
		for every $l \in \ldb 1, a_{n+2} -1 \rdb$ since
		\begin{equation}\label{eqLemIERV1}
		0 \;< \lfp q_{n+2}\alpha \rfp_{2} \;< \lfp k\alpha \rfp_{2}\le \lfp q_{n} \alpha \rfp_{2} 
		\end{equation}
		with $k \;< q_{n+2}$ if and only if $k = q_{n}+ l\cdot q_{n+1}$ for some $l \in \ldb 0, a_{n+2}-1 \rdb$. The claim is proved at the end of the proof. Thus, for every  $l \in \ldb 1, a_{n+2} -1 \rdb$, we have 
		\begin{align*}
		\lc q_{n} + l\cdot q_{n+1} \rc^{\beta}\cdot \lfp \lc q_{n} + l\cdot q_{n+1} \rc \cdot \alpha \rfp_{2}  &= \lc q_{n} + l\cdot q_{n+1} \rc^{\beta}\cdot \lc \lfp  q_{n}\alpha \rfp_{2} + l \lfp  q_{n+1} \alpha \rfp_{2} \rc \\
		&\ge \lc a_{n+2}\cdot l \cdot q_{n} \rc^{\beta} \cdot \lc { a_{n+2} - l\over a_{n+2}} \rc \cdot  \lfp q_{n}\alpha \rfp_{2} \\
		&\ge  a_{n+2}^{\beta -1} \cdot l^{\beta}\cdot \lc a_{n+2} - l\rc \cdot q_{n}^{\beta} \cdot \lfp q_{n}\alpha \rfp_{2} \\
		&\ge  q_{n}^{\beta}\cdot \lfp q_{n} \alpha \rfp_{2} .
		\end{align*}
		It remains to prove claim \eqref{eqLemIERV1} in order to complete the proof of the lemma.
		
		Fix $k \in \ldb 1, q_{n+2} -1 \rdb$ such that $ \ldav k\alpha \rdav \le \ldav q_{n}\alpha \rdav$. The Ostrowski expansion of $k$ is of the form  $ k = \sum_{j=0}^{n+1} e_{j}\cdot q_{j} $. Let $m$ be the minimal natural number in $\ldb 0, n+1\rdb$ such that
		$ e_{m} \ge 1$.  If $m \le n-2$, then, from equation \eqref{eqContinuedFractionsIntervals}, it can easily be deduced that
		$ \ldav \sum_{j=m}^{n+1} e_{j}\cdot q_{j} \cdot \alpha \rdav \;> \ldav q_{j+2} \alpha \rdav \ge \ldav q_{n} \alpha \rdav. $
		This is a contradiction.  If $m = n-1$, then, it cannot hold that $e_{m} \ge 2$ as otherwise, from equation \eqref{eqContinuedFractionsIntervals}, one obtains that $\ldav k\alpha \rdav \ge \ldav q_{n-1} \alpha \rdav $. Thus, $e_{n-1} =1$, which implies from the definition of the Ostrowski expansion that $e_{n} \le a_{n+2} -1$. However, one has that 
		$$\ldav \lc q_{n-1} +e_{n}q_{n} + e_{n+1} q_{n+1} \rc \alpha \rdav  \ge  \ldav \lc q_{n-1} + e_{n} q_{n} \rc \alpha \rdav \underset{ (e_{n}\le a_{n+2} -1 )}{\;>}  \ldav q_{n} \alpha \rdav  , $$
		which yields again a contradiction.  If $m = n$ then $e_{n+1} \le a_{n+2} -1$. We have that $e_{n} \le 1$ since otherwise  $ \ldav \lc e_{n}q_{n} + e_{n+1} q_{n+1} \rc \alpha \rdav \;> \ldav q_{n}\alpha \rdav $. Therefore, we have proved that if $\ldav k\alpha \rdav \le \ldav q_{n}\alpha \rdav $, then $k \in \bigcup_{l=1}^{a_{n+2} -1} \lfp q_{n} +lq_{n+1} \rfp\cup \lfp l q_{n+1} \rfp $. Finally, one has  that for every $l \in \ldb 1, a_{n+2} -1 \rdb$, $\lfp lq_{n+1}\cdot \alpha \rfp_{2} \;< 0 $. Therefore, inequality \eqref{eqLemIERV1} holds if and only if $ k = q_{n} + l q_{n+1} $ with \linebreak $ l\in \ldb 1, a_{n+2} -1 \rdb$. The claim is established, which completes the proof of Lemma \ref{LemIrrationalityEvaluationRationalValues}.  	
	\end{proof}
	This concludes the proof of Theorem \ref{TheorIrrationalityEvaluationRationalValues}.   \hfill $\blacksquare$
\end{proof}

We  end this section by showing that the quantities $\mu_{\pm}(\alpha,\beta,\rho)$ (cf. Definition \ref{DefSignedIrrationalityEvaluation}) cannot be replaced in the statements of Theorems \ref{TheorOscillatingSequence} and \ref{TheorIrrationalityEvaluationRationalValues} with $\mu(\alpha,\beta,\rho)$ (cf. Definition \ref{DefSignedIrrationalityEvaluation}).

\begin{proposition}
	Given $\beta \ge 1$ and a rational number $\rho$, there exists a real $\alpha$ such that $\mu_{+}\lc \alpha,\beta,\rho \rc = 0$ and $\mu_{-}(\alpha,\beta,\rho) \;> 0$. Conversely, there exists a real $\alpha$  such that $\mu_{-}(\alpha,\beta,\rho) =0$ and $\mu_{+}(\alpha,\beta,\rho) \;>0$.	
\end{proposition}

\begin{proof}
	From Theorem \ref{TheorIrrationalityEvaluationRationalValues} it is enough to prove the claim when  $\rho=0$.
	
	Let $\alpha = \lb a_{0}; a_{1},a_{2},... \rb \in \R\backslash\Q$  be an irrational number
	whose partial quotients will be defined recursively.  Set  
	$$ y_{n} = \lb 0; a_{n+1}, a_{n+2},...\rb \in [0,1), \quad  n \in \N_{0} .$$
	Let  $ \lc {p_{n} / q_{n} } \rc_{n \in \N }$ be the sequence of convergents of $\alpha$. A standard analysis of the continued fraction expansion of $\alpha$ yields that
	$$ \lfp q_{n-1}\cdot \alpha \rfp_{2} = { (-1)^{n-1} \over q_{n} + y_{n}\cdot q_{n-1} }  \asymp { (-1)^{n-1} \over q_{n} } \asymp { (-1)^{n-1} \over a_{n}\cdot q_{n-1} } \cdotp  $$
	Therefore,
	$$  q_{n-1} \cdot \lfp q_{n-1}\cdot \alpha \rfp_{2}  \asymp { (-1)^{n-1} \over a_{n} } \cdotp $$
	Define the sequence $\lc a_{n} \rc_{n \in \N} $ as follows: for odd $n\in \N$, choose
	$ a_{n} = \lf n\cdot q_{n-1}^{\beta -1} \rf$  and for even $n \in \N$, choose $1 \le a_{n} \le C$ for some arbitrary predefined positive constant $C \ge 1$.  Then, $ \mu_{+}(\alpha,\beta) = 0  $ and $  \mu_{-}(\alpha,\beta) \;> 0 $.  \hfill $\blacksquare$	
	
\end{proof}

\section{Proof of Theorem \ref{TheorBohrSets}}\label{SectionBohrSets}

\begin{proof}[Theorem \ref{TheorBohrSets}]
	For each $j\ge 0$, write $e_{j}= e_{j}(\rho)$ and set $\mathfrak{D}$ as defined in \eqref{eqTheorBSset}. Also, let $\lc w_{k} \rc_{k\in\N}$ be the sequence defined in \eqref{eqFractionalParts}.  We prove that $ \underset{ \substack{k\to +\oo, \\ k\not\in \mathfrak{D} } }{\lim} \lav w_{k} \rav = +\oo$. This, in turn, implies that the sequence $\lc w_{k} \rc_{k\in\N}$ and the subsequence $\lc w_{k} \rc_{k \in \mathfrak{D}}$ have the same finite limit points.
	
	Fix $h\in \R$. By assumption, given $n_{h}\in \N$ large enough and given $n \ge n_{h}$, it holds that $q_{n}^{\beta}\cdot \ldav q_{n}\alpha \rdav \ge 4|h|$. Fix $n\ge n_{h}$ and let $k \in \ldb \kk_{n}, \kk_{n+1} \rdb$.  If $ k \in \ldb \kk_{n} + q_{n+1} \rdb \backslash \mc{N}_{\rho} \lc \kk_{n}+ q_{n+1}, \alpha, { \ldav q_{n}\alpha \rdav \over 1 + e_{n}^{\beta}} \rc $,  then from the definitions of the inhomogeneous Bohr set \eqref{eqBohrSets} and of the sequence \eqref{eqOstrowskiSequence}, one obtains that
	$$ \lav w_{k} \rav \ge \kk_{n}^{\beta} \cdot { \ldav q_{n}\alpha \rdav \over 1+ e_{n}(\rho)^{\beta} }  \ge   e_{n}(\rho)^{\beta} q_{n}^{\beta} \cdot { \ldav q_{n}\alpha \rdav \over 1+ e_{n}(\rho)^{\beta} }  \ge  2|h|  .$$
	Similarly, if $k \in \ldb \kk_{n} + q_{n+1} , \kk_{n+1} \rdb \backslash \mc{N}_{\rho}\lc \kk_{n+1}, \alpha, \ldav q_{n+1} \alpha \rdav \rc $, then 
	$$ \lav w_{k} \rav \ge q_{n+1}^{\beta} \cdot \ldav q_{n+1}\alpha \rdav  \ge    4|h|  ,$$	
	whence the claim.

	We now prove inclusions \eqref{eqTheorBSInclusions} in the statement of Theorem \ref{TheorBohrSets}.	It follows from the definition of the Bohr set \eqref{eqBohrSets} that
	$$ \mc{N}_{\rho}\lc \kk_{n+1},\alpha, \ldav q_{n+1} \alpha \rdav \rc\cap \ldb 1, \kk_{n} \rdb  \quad  \sub \quad \mc{N}_{\rho}\lc \kk_{n},\alpha, \ldav q_{n} \alpha \rdav \rc $$
	and 
	$$ \mc{N}_{\rho}\lc \kk_{n} + q_{n+1},\alpha, { \ldav q_{n}\alpha \rdav \over 1 + e_{n}^{\beta}} \rc \cap \ldb 1, \kk_{n} \rdb  \quad \sub   \quad \mc{N}_{\rho}\lc \kk_{n},\alpha, \ldav q_{n} \alpha \rdav \rc .$$
	Therefore, it is enough to show, on the one hand that
	\begin{equation}\label{eqTheorBS1}
	\lfp \kk_{n+1} \rfp \quad \sub  \quad \mc{N}_{\rho}\lc \kk_{n+1}, \alpha, \ldav q_{n+1}\alpha \rdav \rc \cap \ldb \kk_{n}+ q_{n+1} , \kk_{n+1} \rdb 
	\end{equation}
	and
	\begin{equation}\label{eqTheorBS2}
	\mc{N}_{\rho}\lc \kk_{n+1}, \alpha, \ldav q_{n+1}\alpha \rdav \rc \cap \ldb \kk_{n}+ q_{n+1} , \kk_{n+1} \rdb \quad  \sub \quad \kk_{n} +  \bigcup_{l =0}^{2} \lfp \lc e_{n+1}- l \rc \cdot q_{n+1} \rfp
	\end{equation} 
	and, on the other, that 
	\begin{equation}\label{eqTheorBS3}
	\begin{split}
	\mc{N}_{\rho}\lc \kk_{n}+ q_{n+1}, \alpha , { \ldav q_{n}\alpha \rdav \over 1+ e_{n}^{\beta}} \rc \cap \ldb \kk_{n}, \kk_{n} + q_{n+1} \rdb	\quad  \sub \quad \kk_{n} + \bigcup_{l=0}^{1} \lfp (l+1)q_{n}\rfp \cup \lfp q_{n+1} - l q_{n} \rfp .
	\end{split}
	\end{equation}

	As far as inclusion \eqref{eqTheorBS1} is concerned, it is easily seen that, for every $n\in \N_{0}$, it holds that $\kk_{n+1} \in  \mc{N}_{\rho}\lc \kk_{n+1},\alpha, \ldav q_{n+1} \alpha \rdav \rc$. As for inclusion \eqref{eqTheorBS2}, there is nothing to prove if $e_{n+1} =0$. Therefore, without loss of generality, assume that  $e_{n+1} \ge 1$ and $k \in \mc{N}_{\rho}\lc \kk_{n+1}, \alpha, \ldav q_{n+1}\alpha \rdav \rc \cap \ldb \kk_{n}+ q_{n+1} , \kk_{n+1} \rdb$. Set 
	$s_{n+1} = \sum_{j= n+2}^{+\oo} e_{j}(\rho) \lfp q_{j}\alpha \rfp_{2} $. Inequality \eqref{eqOstrowskiResidue} yields that $\lav s_{n+1} \rav \le \ldav q_{n+1}\alpha \rdav$.
	Since from the triangle inequality, 
	$$ \ldav k\alpha - e_{n+1}\lfp q_{n+1}\alpha\rfp_{2} \rdav \le \ldav k\alpha - e_{n+1} \lfp q_{n+1}\alpha \rfp_{2} - s_{n+1} \rdav  + \ldav s_{n+1} \rdav \le 2\ldav q_{n+1} \alpha \rdav ,$$   one obtains that 
	\begin{align*}
	 \mc{N}_{\rho}\lc \kk_{n+1}, \alpha , \ldav q_{n+1}\alpha \rdav \rc  \cap \ldb \kk_{n}+q_{n+1}, \kk_{n+1} \rdb \quad \sub \quad \kk_{n} + \mc{N}_{0}\lc e_{n+1}q_{n+1}, \alpha, 2\cdot \ldav q_{n+1} \alpha \rdav \rc .
	 \end{align*} \normalsize
	In turn, this easily implies that 	
	$$ \mc{N}_{\rho}\lc e_{n+1}q_{n+1}, \alpha, 2\cdot \ldav q_{n+1} \alpha \rdav \rc \quad \sub  \quad \lc \bigcup_{l =0}^{2} \lfp \lc e_{n+1}- l \rc \cdot q_{n+1} \rfp \rc \cup \lfp  q_{n+1} - q_{n} \rfp . $$
	Furthermore, $ q_{n+1} - q_{n} \not\in \ldb q_{n+1}, e_{n+1} q_{n+1} \rdb$, which gives the inclusion in \eqref{eqTheorBS2}.  
	
	As for the inclusion in \eqref{eqTheorBS3}, we have that $\mc{N}_{\rho}\lc \kk_{n}+ q_{n+1}, \alpha , { \ldav q_{n}\alpha \rdav \over 1+ e_{n}^{\beta}} \rc \sub \kk_{n} +\mc{N}_{0}\lc  q_{n+1}, \alpha , 2 \ldav q_{n}\alpha \rdav  \rc$. It follows easily that
	\begin{align*}
	\begin{split}
	\mc{N}_{\rho}\lc q_{n+1}, \alpha , 2 \ldav q_{n}\alpha \rdav  \rc  \quad \sub  \quad
	\lfp  q_{n}, 2 q_{n},  q_{n+1} - q_{n}, q_{n+1} \rfp .
	\end{split}
	\end{align*}
	Therefore,
	$$ 	\lfp \kk_{n} \rfp_{n \in \N} \quad \sub \quad \mathfrak{D} $$
	and
	$$ \mathfrak{D} \quad  \sub \quad \bigcup_{n=0}^{+\oo} \lc \kk_{n} + \bigcup_{l =0}^{2} \lfp \lc e_{n}-l \rc \cdot q_{n+1} \rfp \rc \cup \lc \kk_{n} + \bigcup_{l=0}^{1} \lfp (l+1)q_{n}, q_{n+1} - l q_{n} \rfp \rc    . $$
	
	The proof is complete.  \hfill $\blacksquare$
\end{proof}

This work leaves open the question of determining the density properties of the oscillating sequence \eqref{eqOscillatingSequences} defined by more general growth functions than those of the form \eqref{eqGrowthRate}.


\begin{thebibliography}{1}
	
	
		\bibitem{AdiceamRationalApprximationArithmeticProgressions}
	Faustin Adiceam,
	\newblock{Rational Approximation and Arithmetic Progressions,}
	\newblock{ \textit{Int. J. Number Theory}, 11, No.2, p. 451-486, (2015).}
	
	\bibitem{BerendDistributionModulo1ofsomeOscillatingSequences1}
	Daniel Berend, Grigori Kolesnik,
	\newblock{Distribution Modulo 1 Of Some Oscillating Sequences,}
	\newblock{\textit{Israel Journ. of Mathematics}, Vol. 71, No. 2, (1990).}
	
	\bibitem{BerendDistributionModulo1ofsomeOscillatingSequences2}
	Daniel Berend,Michael D. Boshernitzan, Grigori Kolesnik,
	\newblock{Distribution Modulo 1 Of Some Oscillating Sequences, II,}
	\newblock{\textit{Israel Journ. of Mathematics}, Vol. 92, p. 125-147,  (1995).}
	
	\bibitem{BerendDistributionModulo1ofsomeOscillatingSequences3}
	Daniel Berend,Michael D. Boshernitzan, Grigori Kolesnik,
	\newblock{Distribution Modulo 1 Of Some Oscillating Sequences. III,}
	\newblock{\textit{Acta Math. Hungar.} Vol. 95, No. (1-2), p. 1-20,  (2002).}
	
	
	\bibitem{BeresnevichSumsOfReciprocals}
	V. Beresnevich, A. Haynes, S. Velani,
	\newblock{Sums of Reciprocals of Fractional Parts and Multiplicative Diophantine Approximations,}
	\newblock{ \textit{Memoirs of the American Mathematical Society}, Vol. 263, No. 1276, (2020).}
	
	\bibitem{ChowBohrSetsAndMultiplicativeDiophantineApproximation}
	Sam Chow,
	\newblock{Bohr Sets And Multiplicative Diophantine Approximation,}
	\newblock{\textit{Duke Math. J.}, Vol. 167, No. 9, p. 1623-1642, (2018).}
	
	
	\bibitem{MathStackExchangeDensityOfOscillatingSequences}
	Mathematics Stack Exchange,
	\newblock{Is $n\cdot \sin(n)$ dense on the Real Line?}
	\newblock{Available at: {https://math.stackexchange.com/questions/221018/is-n-sin-n-dense-on-the-real-line}}
	
	
	\bibitem{TaoBohrSets}
	Terence Tao,
	\newblock{Continued Fractions, Bohr Sets, and the Littlewood conjecture,}
	\newblock{Available at: {https://terrytao.wordpress.com/2012/01/03/continued-fractions-bohr-sets-and-the-littlewood-conjecture/}}
	

 
 
 
 



\end{thebibliography}
\end{document}